\documentclass[11pt]{article}%
\usepackage{amssymb}
\usepackage{euler}
\usepackage{amsfonts}%
\usepackage{amsmath}%
\usepackage{graphicx}
\providecommand{\U}[1]{\protect\rule{.1in}{.1in}}
\newtheorem{theorem}{Theorem}

\newtheorem{corollary}[theorem]{Corollary}

\newtheorem{lemma}[theorem]{Lemma}

\newtheorem{proposition}[theorem]{Proposition}

\newenvironment{proof}[1][Proof]{\noindent\textbf{#1.} }{\ \rule{0.5em}{0.5em}}

\begin{document}

\title{\textsf{Monotone convergence theorems equivalent to Markov's principle}}
\author{\textsf{Douglas S. Bridges}}
\maketitle

\begin{abstract}%
\noindent
\textsf{The notions of provisional, negative, and apparent convergence to }$0
$\textsf{\ are introduced. It is then shown that Markov's principle is
equivalent, in Bishop-style constructive mathematics, to the statement `every
decreasing sequence of real numbers provisionally convergent to }%
$0$\textsf{\ actually converges to }$0$\textsf{', and that this equivalence
holds with `provisionally' replaced by `negatively'. Finally, apparent
convergence and convergence are related by means of the anti-Specker
principle.}

\end{abstract}%

\bigskip
\normalfont\sf

It is well known that the monotone convergence theorem for sequences is
constructively equivalent to the omniscience principle

\begin{quote}
\textbf{LPO}: \ For each binary sequence $\left(  a_{n}\right)  _{n\geqslant1}
$, either $a_{n}=0$ for all $n$ or else there exists $n$ with $a_{n}=1,$
\end{quote}%

\noindent
and is therefore essentially nonconstructive. In this note, written entirely
within Bishop-style constructive mathematics (\textbf{BISH),}\footnote{%
\normalfont\sf
By `Bishop-style constructive mathematics' we mean mathematics as presented in
\cite{Bishop,BB,BVbook}. Essentially, this is mathematics carried out with
intuitionistic logic and within a set- or type-theoretic framework such as
those found in \cite{Balps,ML,Rathjen}. Our paper is written entirely in that
constructive setting.} we discuss the relation between three constructively
weaker forms of the monotone convergence theorem, Markov's principle (a
weakening of \textbf{LPO}),

\begin{quote}
\textbf{MP}: \ For each binary sequence $\left(  a_{n}\right)  _{n\geqslant1}%
$, if it is impossible that $a_{n}=0$ for all $n$, then there exists $n$ with
$a_{n}=1$,
\end{quote}%

\noindent
and the \textbf{anti-Specker property}\footnote{%
\normalfont\sf
This property is equivalent, over \textbf{BISH}, to Brouwer's fan theorem
\textbf{FT}$_{c}$ for so-called \textquotedblleft$c$-bars\textquotedblright%
\ \cite{BBerg}, and is classically equivalent to the sequential compactness
(Bolzano--Weierstra\ss \ property) of $\left[  0,1\right]  $.} for the
interval $\left[  0,1\right]  $,

\begin{quote}
\textbf{AS}$_{\left[  0,1\right]  }$: \ If $\mathbf{x}$ is a sequence in
$\mathbb{R}$ that is eventually bounded away from each point of $\left[
0,1\right]  $, then $\mathbf{x}$ is eventually bounded away uniformly from the
interval $\left[  0,1\right]  $.
\end{quote}%

\bigskip

For our first result, we define a sequence $\mathbf{x}\equiv\left(
x_{n}\right)  _{n\geqslant1}$ in $\mathbb{R}$ to be \textbf{provisionally
convergent to }$0$ if the following holds: for all $t\neq0$ and $\varepsilon
>0$, if $\left\vert x_{n}-t\right\vert <\varepsilon$ eventually (that is, for
all sufficiently large $n$), then $\left\vert t\right\vert \leqslant
\varepsilon$.

\begin{proposition}
\label{1}A sequence in $\mathbb{R}$ converges to $0$ if and only if it
converges to a limit and is provisionally convergent to $0$.
\end{proposition}

\begin{proof}
Suppose that $\mathbf{x}\equiv\left(  x_{n}\right)  _{n\geqslant1}$ converges
to $0$, that $\varepsilon>0$, and that $\left\vert x_{n}-t\right\vert
<\varepsilon$ for all sufficiently large $n$. If $\left\vert t\right\vert
>\varepsilon$, then for all sufficiently large $n$ we have $\left\vert
x_{n}\right\vert <\left\vert t\right\vert -\varepsilon$ and therefore%
\[
\left\vert t\right\vert \leqslant\left\vert x_{n}-t\right\vert +\left\vert
x_{n}\right\vert <\varepsilon+\left(  \left\vert t\right\vert -\varepsilon
\right)  =\left\vert t\right\vert ,
\]
which is absurd. Hence $\left\vert t\right\vert \leqslant\varepsilon$, and
therefore $\mathbf{x}$ is provisionally convergent.

Now let $\mathbf{x}$ converge to a limit $t\in\left[  0,1\right]  $ and be
provisionally convergent to $0$. Suppose that $\left\vert t\right\vert >0$;
then $\left\vert x_{n}-t\right\vert <\left\vert t\right\vert /2$ eventually,
so, by the provisional convergence of $\mathbf{x}$ to $0$, $\left\vert
t\right\vert \leqslant\left\vert t\right\vert /2$, which is absurd. Hence
$t=0$.
\end{proof}%

\bigskip

In the recursive model of \textbf{BISH} there exists a sequence $\mathbf{s}%
\equiv\left(  s_{n}\right)  _{n\geqslant1}$ in $\left[  0,1\right]  $ that is
eventually bounded away from each point of $\left[  0,1\right]  $. (Such a
sequence is known as a \textbf{Specker sequence}; for more on these, see
Chapter 3 of \cite{BR} or Specker's original paper \cite{Specker}.) In
particular, for each $x\in\left[  0,1\right]  $, the sequence $\left(
x-s_{n}\right)  _{n\geqslant1}$ is neither convergent to $0$ nor provisionally
convergent to $0$. To see that the latter holds, choose $\varepsilon>0$ such
that $\left\vert x-s_{n}\right\vert >\varepsilon$ eventually. If $t\neq0$ and
$\left\vert x-s_{n}-t\right\vert <\varepsilon/2 $, then $\left\vert
t\right\vert \geq\left\vert x-s_{n}\right\vert -\varepsilon/2>\varepsilon/2$.

\begin{lemma}
\label{130511a3}Let $x\in\left[  0,1\right]  $ satisfy $\lnot\left(
x=1\right)  $. Then the sequence $\left(  x^{n}\right)  _{n\geqslant1}$ is
provisionally convergent to $0$. If it converges to $0$, then $x<1$.
\end{lemma}

\begin{proof}
Let $\varepsilon>0$, and suppose that $\left\vert x^{n}-t\right\vert
<\varepsilon$ for all $n\geqslant N$. If $x\neq1$, then $x^{n}\rightarrow0 $,
so $\left\vert t\right\vert \leqslant\varepsilon$; from which it follows that
if $\left\vert t\right\vert >\varepsilon$, then $\lnot\left(  x\neq1\right)
$, and therefore $x=1$, a contradiction. Hence, in fact, $\left\vert
t\right\vert \leqslant\varepsilon$. On the other hand, if $x^{n}\rightarrow0$,
then there exists $N$ such that $x^{n}<1$, whence $x=\sqrt[n]{x^{n}}<1$.
\end{proof}

\begin{proposition}
\label{2}The following are equivalent over\emph{\ \textbf{BISH}}:

\begin{itemize}
\item[\emph{(i)}] \emph{\textbf{MP}}.

\item[\emph{(ii)}] Every decreasing sequence of nonnegative real numbers that
is provisionally convergent to $0$ actually converges to $0$.
\end{itemize}
\end{proposition}

\begin{proof}
Assume \textbf{MP}. Let $\mathbf{x\equiv}\left(  x_{n}\right)  _{n\geqslant1}
$ be a decreasing sequence of nonnegative numbers converging provisionally to
$0$. Given $\varepsilon>0$, assume that $x_{n}>\varepsilon$ for all $n$.
Setting $n_{0}=1$, construct an increasing binary sequence $\left(
\lambda_{n}\right)  _{n\geqslant1}$ such that%
\begin{align*}
\lambda_{n}=0  & \Rightarrow\forall_{k\leqslant n}\left(  \left\vert
x_{k}-x_{1}\right\vert <\varepsilon\right)  ,\\
\lambda_{n}=1-\lambda_{n-1}  & \Rightarrow\left\vert x_{n}-x_{1}\right\vert
>\varepsilon/2.
\end{align*}
Since $\mathbf{x}$ is provisionally convergent to $0$, if $\lambda_{n}=0$ for
all $n$, then $\left\vert x_{1}\right\vert <\varepsilon$, which is absurd. It
follows from \textbf{MP} that there exists $n_{1}$ such that $\left\vert
x_{n_{1}}-x_{1}\right\vert >\varepsilon/2$. Repeating this with $x_{n_{1}}$
replacing $x_{1}$, and with the sequence $\left(  x_{n_{1}+1},x_{n_{1}%
+2},\ldots\right)  $ replacing $\mathbf{x}$, we obtain $n_{2}>n_{1}$ such that
$\left\vert x_{n_{2}}-x_{n_{1}}\right\vert >\varepsilon/2$. Carrying on in
this way, we construct a strictly increasing sequence $\left(  n_{k}\right)
_{k\geqslant1}$ such that $\left\vert x_{n_{k}}-x_{n_{k-1}}\right\vert
>\varepsilon/2$ for each $k\geqslant1$. Choosing a positive integer%
\[
K\geqslant2\left(  1+\frac{x_{1}}{\varepsilon}\right)  ,
\]
since the sequence $\mathbf{x}$ is decreasing we now see that%
\begin{align*}
x_{n_{K}}  & =x_{n_{1}}+\sum_{k=2}^{K}\left(  x_{n_{k}}-x_{n_{k-1}}\right) \\
& \leqslant x_{1}-\sum_{k=2}^{K}\left\vert x_{n_{k}}-x_{n_{k-1}}\right\vert
<x_{1}-\left(  K-2\right)  \frac{\varepsilon}{2}\leqslant0,
\end{align*}
a contradiction from which we conclude that%
\[
\lnot\forall_{n}\left(  x_{n}>\varepsilon\right)  .
\]
Applying \textbf{MP} once more, we deduce that there exists $\nu$ such that
$x_{\nu}<2\varepsilon$. Hence $0\leqslant x_{n}\leqslant x_{\nu}<2\varepsilon$
for all $n\geqslant\nu$. Since $\varepsilon>0$ is arbitrary, the proof that
(i) implies (ii) is complete.

Now assume (ii), and consider $a\in\left[  0,1\right]  $ with $\lnot\left(
a=0\right)  .$ By Lemma \ref{130511a3}, the sequence $\left(  \left(
1-a\right)  ^{n}\right)  _{n\geqslant1}$ is provisionally convergent to $0$,
and if it converges to $0$, then $1-a<1$ and therefore $a\neq0$. Thus the
statement in question implies that%
\[
\forall_{x\in\mathbb{R}}\left(  \lnot(x=0)\Rightarrow x\neq0\right)  ,
\]
which is equivalent to \textbf{MP}. Hence (ii) implies (i).
\end{proof}%

\bigskip

There is another weak notion of convergence that, under \textbf{MP}, we can
relate to convergence: we say that a sequence $\mathbf{x}$ of real numbers is
\textbf{negatively convergent to }$0$ if for each $\varepsilon>0$ it is
impossible that $\left\vert x_{n}\right\vert >\varepsilon$ for infinitely many
$n$.

\begin{proposition}
\label{3005a3}If a sequence in $\mathbb{R}$ is negatively convergent to $0$,
then it is provisionally convergent to $0$.
\end{proposition}

\begin{proof}
Let the sequence $\mathbf{x}$ be negatively convergent to $0$. Given
$\varepsilon>0$, suppose that $\left\vert x_{n}-t\right\vert <\varepsilon$ for
all $n\geqslant N$. If $\left\vert t\right\vert >\varepsilon$, then for
$n\geqslant N$ we have $\left\vert x_{n}\right\vert >\left\vert t\right\vert
-\varepsilon>0$, which contradicts the negative convergence of $\mathbf{x}$ to
$0$. Hence $\left\vert t\right\vert \leqslant\varepsilon$. Since
$\varepsilon>0$ is arbitrary, we conclude that $\mathbf{x}$ is provisionally
convergent to $0$.
\end{proof}%

\bigskip

To see that the converse of Proposition \ref{3005a3} is false, take
$x_{n}\equiv\left(  -1\right)  ^{n-1}$. Then $\left\vert x_{n}\right\vert =1$
for each $n$, so $\mathbf{x}$ is not negatively convergent to $0$. However, it
is provisionally convergent to $0$: for if $\left\vert x_{n}-t\right\vert
<\varepsilon<\left\vert t\right\vert $ for all sufficiently large $n$, then
$\left\vert 1-t\right\vert <\varepsilon$ and $\left\vert 1+t\right\vert
<\varepsilon$, so%
\[
1<\min\left\{  \varepsilon-t,\varepsilon+t\right\}  <0,
\]
a contradiction.

\begin{proposition}
\label{3005a1}A sequence of real numbers converges to $0$ if and only if it
converges to a limit and is negatively convergent to $0$.
\end{proposition}

\begin{proof}
It is trivial that if $\mathbf{x}\equiv\left(  x_{n}\right)  _{n\geqslant1}$
converges to $0$, then it is negatively convergent to $0$. Conversely, if
$\mathbf{x}$ converges to a limit $t$ and is negatively convergent to $0$,
then it follows from Propositions \ref{3005a3} and \ref{1} that $t=0$.
\end{proof}

\begin{proposition}
\label{3005a2}The following are equivalent over\emph{\ \textbf{BISH}}:

\begin{itemize}
\item[\emph{(i)}] \emph{\textbf{MP}}.

\item[\emph{(ii)}] Every decreasing sequence of nonnegative real numbers that
is negatively convergent to $0$ actually converges to $0$.
\end{itemize}
\end{proposition}

\begin{proof}
Propositions \ref{3005a3} and \ref{2} show that (i) implies (ii). Now suppose,
conversely, that every decreasing sequence of nonnegative numbers that is
negatively convergent to $0$ actually converges to $0$. Consider $a\in\left[
0,1\right]  $ with $\lnot\left(  a=0\right)  $, and take $x_{n}\equiv
(1-a)^{n}$. Then $\mathbf{x}$ is a decreasing sequence of nonnegative numbers.
Suppose that there exist $\varepsilon>0$ and $n_{1}<n_{2}<\cdots$ such that
$x_{n_{k}}\geqslant\varepsilon$ for all $k$. If $a>0$, this is impossible; so
we must have $a=0$, a contradiction. We conclude that $\mathbf{x}$ is
negatively convergent to $0$. Hence it converges to $0$, so we can find $n$
such that $\left(  1-a\right)  ^{n}<1$, from which it follows that $a\neq0$.
Thus (ii) implies \textbf{MP}.
\end{proof}

Now recall the \textbf{anti-Specker property,}

\begin{quote}
\textbf{AS}$_{\left[  0,1\right]  }$: \ If $\mathbf{x}$ is a sequence in
$\mathbb{R}$ that is eventually bounded away from each point of $\left[
0,1\right]  $, then $\mathbf{x}$ is eventually bounded away uniformly from the
interval $\left[  0,1\right]  $.
\end{quote}%

\noindent
This property is equivalent, over \textbf{BISH}, to Brouwer's fan theorem
\textbf{FT}$_{c}$ for so-called \textquotedblleft$c$-bars\textquotedblright%
\ \cite{BBerg}, and is classically equivalent to the sequential compactness
(Bolzano--Weierstra\ss \ property) of $\left[  0,1\right]  $. We bring
\textbf{AS}$_{[0,1}]$ into play by introducing yet another convergence notion:
we call a \emph{decreasing} sequence $\mathbf{x}$ of nonnegative numbers
\textbf{apparently convergent to} $0$ if it is eventually bounded away from
each point of the interval $(0,x_{1}]$.

\begin{proposition}
\label{1207a1}\emph{\textbf{AS}}$_{\left[  0,1\right]  }\vdash$ \ Let
$\mathbf{x}\equiv\left(  x_{n}\right)  _{n\geqslant1}$ be a decreasing
sequence of nonnegative real numbers that is apparently convergent to $0$.
Then $\mathbf{x}$ is negatively convergent to $0$.
\end{proposition}

\begin{proof}
Given $\varepsilon>0$, suppose that there exists a strictly increasing
sequence $\left(  n_{k}\right)  _{k\geqslant1}$ of positive integers such that
$x_{n_{k}}>\varepsilon$ for all $k$. Then $x_{1}>\varepsilon$, and $\left(
x_{n_{k}}\right)  _{k\geq1}($ is bounded away from\ each point of the interval
$[0,\varepsilon)$. Since, being a subsequence of the apparently convergent
sequence $\mathbf{x}$, it is eventually bounded away from each point of
$(0,x_{1}]$, it follows that $\left(  x_{n_{k}}\right)  _{k\geqslant1}$ is
eventually bounded away from each point of $\left[  0,x_{1}\right]  $.
Applying \textbf{AS}$_{[0,1]}$, we now see that it is eventually bounded away
from the entire interval $\left[  0,x_{1}\right]  $, which is absurd since
$\mathbf{x}$ is nonnegative and decreasing. Thus it is impossible that
$x_{n}>\varepsilon$ infinitely often.%

\hfill

\end{proof}%

\bigskip

Note that in the recursive model of \textbf{BISH}, any decreasing Specker
sequence in $\left[  0,1\right]  $ is eventually bounded away from each point
of $\left[  0,1\right]  $ and is therefore apparently convergent---but not
convergent---to $0$. Thus in order to pass from apparent convergence to actual
convergence, we will need to add to \textbf{BISH }some explicitly
non-recursive, presumably intuitionistic, principle. In fact, \textbf{AS}%
$_{\left[  0,1\right]  }$ is enough, provided we also assume \textbf{MP}.

\begin{corollary}
\label{1207a1a}\emph{\textbf{MP} + \textbf{AS}}$_{\left[  0,1\right]  }%
\vdash\ $ A decreasing sequence of nonnegative numbers converges to $0$ if and
only if it is apparently convergent to $0$.
\end{corollary}

\begin{proof}
Let $\mathbf{x}$ be a decreasing sequence of nonnegative numbers. Clearly, if
$\mathbf{x}$ converges to $0$, then it is apparently convergent to $0$.
Conversely, if $\mathbf{x}$ is apparently convergent to $0$, then, by
Proposition \ref{1207a1}, it is negatively convergent to $0$; whence, by
Proposition \ref{3005a2}, it converges to $0$.
\end{proof}%

\bigskip
%

\bigskip

%

\bigskip

\bigskip
%


\noindent
Specker property%

\noindent
\textbf{MR (2010) Class. Nos}: 03F60, 40A05%

\bigskip
%

\bigskip
%

\noindent
\textbf{Author's address}: Department of Mathematics \& Statistics, University
of Canterbury, Private Bag 4800, Christchurch 8140, New Zealand.%

\noindent
\textbf{email}: \texttt{dsbridges.math@gmail.com}

\end{document}